\documentclass{amsart}

\usepackage{amsmath}             
\usepackage{amsfonts}          
\usepackage{amsrefs}

\newtheorem{thm}{Theorem}[section]
\newtheorem{lem}[thm]{Lemma}
\newtheorem{prop}[thm]{Proposition}
\newtheorem{cor}[thm]{Corollary}

\theoremstyle{definition}

\newtheorem{nota}[thm]{Notation}
\newtheorem{rem}[thm]{Remarks}
\newtheorem{exam}[thm]{Example}
\newtheorem{cons}[thm]{Construction}

\begin{document}
\date{\today}
\nocite{*}

\title{Normal Limits of Nilpotent Operators in C$^*$-Algebras}

\author{Paul Skoufranis}
\address{Department of Mathematics, UCLA, Los Angeles, California, USA, 90095-1555}
\email{pskoufra@math.ucla.edu}
\thanks{This research was supported in part by NSERC PGS D3-389187-200.}


\begin{abstract}
We will investigate the intersection of the normal operators with the closure of the nilpotent operators in various C$^*$-algebras.  A complete description of the intersection will be given for unital, simple, purely infinite C$^*$-algebras.  The intersection in AF C$^*$-algebras is also of interest.  In addition, an example of a separable, nuclear, quasidiagonal C$^*$-algebra where every operator is a limit of nilpotent operators will be constructed.
\end{abstract}

\maketitle

\section{Introduction}

As an $n \times n$ matrix has trivial spectrum if and only if it is nilpotent and the set of nilpotent $n \times n$ matrices is closed in the operator topology, in Problem 7 of \cite{Hal} Halmos posed the question, ``Is every quasinilpotent operator (that is, an operator $T$ with spectrum $\{0\}$) on a complex, separable, infinite dimensional Hilbert space the norm limit of nilpotent operators (that is, operators $T$ such that $T^k = 0$ for some $k \in \mathbb{N}$; note every nilpotent operator is automatically quasinilpotent)?"  An affirmative answer to Halmos problem was given in \cite{AV} (also see \cite{AS} and \cite{AFP}).  However, Halmos realized that his problem was incorrect as \cite{Hed} showed that there exists non-quasinilpotent operators that are norm limits of nilpotent operators.  Thus Halmos reposed his question as, ``What is the closure of the nilpotent operators on a complex, separable, infinite dimensional Hilbert space?'' or equivalently, ``What is the closure of all operators on a complex, separable, infinite dimensional Hilbert space with trivial spectrum?"  A complete characterization of the closure of the nilpotent operators was first given in \cite{AFV}:
\begin{thm}[\cite{AFV}, see Theorem 5.1 in \cite{He1} for a proof]
\label{BHchar}
Let $\mathcal{H}$ be a complex, separable, infinite dimensional Hilbert space and let $T$ be a bounded linear operator on $\mathcal{H}$.  Then $T$ is a norm limit of nilpotent operators on $\mathcal{H}$ if and only if the following conditions are satisfied:
\begin{enumerate}
	\item The spectrum of $T$ is connected and contains zero.
	\item The essential spectrum of $T$ is connected and contains zero.
	\item The Fredholm index of $\lambda I_\mathcal{H} - T$ is zero for all $\lambda \in \mathbb{C}$ such that $\lambda I_\mathcal{H} - T$ is semi-Fredholm.
\end{enumerate}
\end{thm}
Herrero performed a significant amount of work in an attempt determine the closure of the nilpotent operators on a complex, separable, infinite dimensional Hilbert space (see \cite{He2}, \cite{He4}, and \cite{He3}).  In particular, before \cite{AFV}, Herrero proved the following specific case of Theorem \ref{BHchar} in an elegant way:
\begin{thm}[Theorem 7 of \cite{He2}, also see Proposition 5.6 in \cite{He1} and Theorem 2 in \cite{Had}]
\label{BHnormchar}
Let $N$ be a normal operator on a complex, separable, infinite dimensional Hilbert space $\mathcal{H}$.  Then the following are equivalent:
\begin{enumerate}
	\item $N$ is a norm limit of nilpotent operators on $\mathcal{H}$.
	\item $N$ is a norm limit of quasinilpotent operators on $\mathcal{H}$.
	\item The spectrum of $N$ is connected and contains zero.
\end{enumerate}
\end{thm}
Due to the existence and elegance of multiple proofs of Theorem \ref{BHnormchar}, it is natural rephrase the above question in the context of C$^*$-algebras; that is, ``Given an arbitrary C$^*$-algebra $\mathfrak{A}$ and a normal operator $N \in \mathfrak{A}$, can simple conditions be given to determine whether $N$ is a norm limit of nilpotent or quasinilpotent operators from $\mathfrak{A}$?"  Although the GNS construction implies $\mathfrak{A}$ can be embedded faithfully into the bounded linear operators on a Hilbert space, Theorem \ref{BHnormchar} does not provided the answer to this question as the image of $\mathfrak{A}$ need not contain the necessary nilpotent or quasinilpotent operators.  However, a solution to this question can be easily obtained in several particular cases.  For example, this question is easily solved for abelian C$^*$-algebras (which have no non-trivial quasinilpotent operators).  In addition, this problem has already been solved for the Calkin algebra:
\begin{thm}[see Theorem 5.34 in \cite{He1}]
\label{Calkin}
Let $\mathcal{B}(\mathcal{H})$ be the bounded linear operators on a complex, separable Hilbert space $\mathcal{H}$, let $\mathfrak{A}$ be the Calkin algebra, let $q : \mathcal{B}(\mathcal{H}) \to \mathfrak{A}$ be the canonical quotient map, and let $T \in \mathcal{B}(\mathcal{H})$.  Then $q(T)$ is a norm limit of nilpotent operators from $\mathfrak{A}$ if and only if the essential spectrum of $T$ is connected and contains zero and the Fredholm index of $\lambda I_\mathcal{H} - T$ is zero for all $\lambda$ such that $\lambda I_\mathcal{H} - T$ is semi-Fredholm.
\end{thm}
More recently in \cite{Sk}, a complete classification of which normal operators were norm limits of nilpotent and quasinilpotent operators was obtained for type I and type III von Neumann algebras with separable predual (with some partial results for type II von Neumann algebras):  
\begin{thm}[Corollary 2.6 in \cite{Sk}]
\label{typeI}
Let 
\[
\mathfrak{M} := L_\infty(X, \mathcal{B}(\mathcal{H})) \oplus \left(\prod_{n\geq 1} \mathcal{M}_n(\mathbb{C}) \overline{\otimes} L_\infty(X_n)\right)
\]
where $(X, \mu)$ and $(X_n, \mu_n)$ are Radon measure spaces and $\mathcal{B}(\mathcal{H})$ is the set of bounded linear operators on a complex, separable, infinite Hilbert space $\mathcal{H}$.  Let $P \in \mathfrak{M}$ be the (central) projection onto $L_\infty(X, \mathcal{B}(\mathcal{H}))$ and let $N \in \mathfrak{M}$ be a normal operator.  Then the following are equivalent:
\begin{enumerate}
	\item $N$ is a norm limit of nilpotent operators from $\mathfrak{M}$.
	\item $N$ is a norm limit of quasinilpotent operators from $\mathfrak{M}$.
	\item $PN = N$ and the spectrum of $N(x)$ is connected and contains zero for almost every $x \in X$.
\end{enumerate}
\end{thm}
\begin{thm}[Theorem 3.2 in \cite{Sk}]
\label{typeIII}
Let $\mathfrak{M}$ be a type III von Neumann algebra with separable predual.  Then there exists a locally compact, complete, separable, metrizable measure space $(X, \mu)$ and a collection of type III factors $(\mathfrak{M}_x)_{x \in X}$ with separable predual such that $\mathfrak{M}$ is a direct integral of $(\mathfrak{M}_x)_{x \in X}$.  If $N \in \mathfrak{M}$ is a normal operator, we may write $N = \int^\oplus_X N_x \,d\mu(x)$ where $N_x \in \mathfrak{M}_x$ is a normal operator $\mu$-almost everywhere.  Then the following are equivalent:
\begin{enumerate}
	\item $N$ is a norm limit of nilpotent operators from $\mathfrak{M}$.
	\item $N$ is a norm limit of quasinilpotent operators from $\mathfrak{M}$.
	\item $N_x$ is a norm limit of nilpotent operators from $\mathfrak{M}_x$ for almost every $x \in X$.
	\item $N_x$ is a norm limit of quasinilpotent operators from $\mathfrak{M}_x$ for almost every $x \in X$.
	\item The spectrum of $N_x$ is connected and contains zero for almost every $x \in X$.
\end{enumerate}
\end{thm}
In particular, the following are necessary requirements for an operator to be a norm limit of quasinilpotent operators and are derived from the facts that the set of invertible elements is an open set in any C$^*$-algebra and the semicontinuity of the spectrum.
\begin{lem}[for a proof, see Lemma 1.3 in \cite{Sk}]
\label{cac0}
Let $\mathfrak{A}$ be a C$^*$-algebra and let $T\in \mathfrak{A}$ be a limit of quasinilpotent operators from $\mathfrak{A}$.  Then the spectrum of $T$ is connected and contains zero.
\end{lem}
In this paper, which is a continuation of \cite{Sk}, we will examine whether conditions can be given to determine when a normal operator is a limit of nilpotent or quasinilpotent operators in various C$^*$-algebras.  
\par
Section \ref{sec:USPI} will examine this question in the context of unital, simple, purely infinite C$^*$-algebras.  As unital, simple, purely infinite C$^*$-algebras have a plethora of projections with particular structure similar to that of von Neumann algebras, a complete solution to our problem will be obtained for said algebras (see Theorem \ref{uspimain}).  In particular, as the Calkin algebra is a unital, simple, purely infinite C$^*$-algebra, Section \ref{sec:USPI} will generalize Theorem \ref{Calkin}.  Section \ref{sec:USPI} will also examine auxiliary questions such as the closure of the span of nilpotent opertors and the distance from a projection to the nilpotent operators in any unital, simple, purely infinite C$^*$-algebra.  
\par
Section \ref{sec:AFALGEBRAS} will examine this question in the context of AF C$^*$-algebras.  AF C$^*$-algebras are one generalization of finite dimensional C$^*$-algebras and thus it is surprising that the closure of nilpotent operators in said algebras is incredible complex.  In particular, Section \ref{sec:AFALGEBRAS} relates the norm closure of the nilpotent operators in AF C$^*$-algebras to the asymptotic behaviour of nilpotent matrices as the dimension of the matrices are allowed to increase and will demonstrate the existence of AF C$^*$-algebras with non-zero normal operators in the closure of the nilpotent operators.  Since the submission of this paper, the author has used one of the main results of Section \ref{sec:AFALGEBRAS}, Theorem \ref{UHFnil}, to obtain additional interesting results along with an additional (yet more complicated) proof of Theorem \ref{uspimain} (see \cite{Sk2}).
\par
Section \ref{sec:READSGEN} will generalize a construction from \cite{Re} to demonstrate that there exists a separable, nuclear, quasidiagonal C$^*$-algebra where every operator is a norm limit of nilpotent operators.  The cone of this C$^*$-algebra is then AF-embeddable and it will be demonstrated this cone has also has the property that every operator is a norm limit of nilpotent operators.
\begin{nota}
The following will be the notation used throughout the paper.
\begin{itemize}
	\item $\mathcal{M}_n(\mathfrak{A})$ - the C$^*$-algebra of $n \times n$ matrices with entries in a C$^*$-algebra $\mathfrak{A}$.
	\item $Nor(\mathfrak{A})$ - the set of normal operators of a C$^*$-algebra $\mathfrak{A}$.
	\item $\mathfrak{A}_{sa}$ - the set of self-adjoint operators of a C$^*$-algebra $\mathfrak{A}$.
	\item $\mathfrak{A}_+$ - the set of positive operators of a C$^*$-algebra $\mathfrak{A}$.
	\item $Nil(\mathfrak{A})$ - the set of nilpotent operators of a C$^*$-algebra $\mathfrak{A}$.
	\item $QuasiNil(\mathfrak{A})$ - the set of quasinilpotent operators of a C$^*$-algebra $\mathfrak{A}$.
	\item $\sigma(T)$ - the spectrum of an operator $T$.
	\item $\mathfrak{A}^{-1}$ - the set of invertible operators of a unital C$^*$-algebra $\mathfrak{A}$.
	\item $\mathfrak{A}^{-1}_0$ - the connected component of the identity in $\mathfrak{A}^{-1}$.
	\item $dist(A, B)$ - the distance between two sets $A$ and $B$ in a normed linear space.
	\item $diag(a_1,\ldots, a_n)$ - the $n \times n$ diagonal matrix with $a_1,\ldots, a_n$ along the diagonal.
\end{itemize}
\end{nota}

\section{Purely Infinite C$^*$-Algebras}
\label{sec:USPI}

In this section we will prove our main result, Theorem \ref{uspimain}, which completely classifies when a normal operator in a unital, simple, purely infinite C$^*$-algebra is a norm limit of nilpotent and quasinilpotent operators.  The main tools of the proof are the existence and equivalence of projections in unital, simple, purely infinite C$^*$-algebras and Lemma \ref{mas} which gives positive matrices of norm one that are asymptotically approximated by nilpotent matrices as we allow the size of the matrices to increase.  In fact, in the case that $\mathfrak{A}^{-1}_0 = \mathfrak{A}^{-1}$, the conditions of Theorem \ref{uspimain} are identical to the conditions of Theorem \ref{BHnormchar}.  This is not a surprise as the proof of Theorem \ref{uspimain} relies only on Lemma \ref{mas} and the structure of the projections in a unital, simple, purely infinite C$^*$-algebra.  In fact, the proof of Theorem \ref{uspimain} can be adapted to prove Theorem \ref{BHnormchar}.  When the proof of Theorem \ref{uspimain} has been completed, we will apply similar techniques to obtain information about the closed span of nilpotent operators and the distance from a fixed projection to the nilpotent operators in unital, simple, purely infinite C$^*$-algebras.  
\par
For completeness we include an outline of the following previously known result.
\begin{lem}[See A1.14 in \cite{AFHV}]
\label{mas}
For each $n \in \mathbb{N}$ there exists a positive matrix $A_n \in \mathcal{M}_n(\mathbb{C})$ with norm one such that $\lim_{n\to\infty} dist(A_n, Nil(\mathcal{M}_n(\mathbb{C}))) = 0$.
\end{lem}
\begin{proof}
Let 
\[
Q'_n := \sum^{n-1}_{j=1} \frac{1}{j} q_n^j \in \mathcal{M}_n(\mathbb{C})
\]
where $q_n \in \mathcal{M}_n(\mathbb{C})$ is the nilpotent Jordan block of order $n$.  It was shown in \cite{Ka} that $\left\|Re(Q'_n)\right\| \leq \frac{\pi}{2}$.  If $Q_n := \frac{\mathrm{i}}{\ln(n)} Q'_n \in \mathcal{M}_n(\mathbb{C})$ and $H_n := Re(Q_n) \in \mathcal{M}_n(\mathbb{C})$, then, by \cite{Ka}, $-I_n \leq H_n \leq I_n$, $Q_n \in Nil(\mathcal{M}_n(\mathbb{C}))$,
\[
|\left\|H_n\right\| - 1| \leq \frac{\ln(2)}{2\ln(n)},\mbox{ and }\left\|H_n - Q_n\right\| \leq \frac{\pi}{2\ln(n)}.
\]  
\par
By normalizing each $H_n$, we obtain self-adjoint matrices $B_n \in \mathcal{M}_n(\mathbb{C})$ with norm one such that 
\[
\lim_{n\to\infty} dist(B_n, Nil(\mathcal{M}_n(\mathbb{C}))) = 0.
\]
For each $n \in \mathbb{N}$ let $A_n := B_n^2$.  Hence $A_n \in \mathcal{M}_n(\mathbb{C})$ is a positive matrix with norm one.  Since the square of any nilpotent matrix is a nilpotent matrix, it is easy to obtain
\[
\lim_{n\to\infty} dist(A_n, Nil(\mathcal{M}_n(\mathbb{C}))) = 0
\]
as desired.
\end{proof}
Although the following two results are not required for the proof of Theorem \ref{uspimain}, we include them here for completeness and for later discussions.  Lemma \ref{distrestr} appears in \cite{Sk} and we include its simple proof for convenience.
\begin{lem}[Lemma 2.3 in \cite{Sk}]
\label{distrestr}
Let $\mathfrak{A}$ be a C$^*$-algebra and let $A \in \mathfrak{A}_+$.  Suppose $A$ is not invertible and $\sigma(A) = \{0 = \lambda_0 < \lambda_1 < \cdots < \lambda_k\}$.  Then
\[
dist(A, QuasiNil(\mathfrak{A})) \geq \frac{1}{2} \max_{1\leq i \leq k} |\lambda_i - \lambda_{i-1}|.
\]
\end{lem}
\begin{proof}
It is easy to see that there exists $x_0,y_0 \in \sigma(A)$ and a Cauchy domain $\Omega$ such that such that $|x_0 - y_0| = \max_{1\leq i \leq k} |\lambda_i - \lambda_{i-1}|$, $0 \notin \Omega$, and
\[
\inf\left\{ \left\|(\lambda I - A)^{-1}\right\|^{-1} \, \mid \, \lambda \in \delta \Omega\right\} = \frac{1}{2}|x_0-y_0|.
\]
If $M \in \mathfrak{A}$ is such that
\[
\left\|A - M\right\| < \inf\left\{ \left\|(\lambda I - A)^{-1}\right\|^{-1} \, \mid \, \lambda \in \delta \Omega\right\}
\]
then, by an application of the Analytic Functional Calculus for Banach Algebras (see Theorem 1.1 in \cite{He1}), $\sigma(M) \cap \Omega \neq \emptyset$.  As $0 \notin \Omega$ and the spectrum of any quasinilpotent operator is $\{0\}$, the result trivially follows.
\end{proof}
\begin{cor}
\label{densespectra}
Let $\{A_n\}_{n\geq 1}$ be the positive matrices of norm one from Lemma \ref{mas}.  For every $m \in \mathbb{N}$ there exists an $N_m \in \mathbb{N}$ such that 
\[
\sigma(A_n)\cap \left[\frac{k}{m}, \frac{k+1}{m}\right) \neq \emptyset
\]
for all $k \in \{0, 1, \ldots, m-1\}$ and for all $n \geq N_m$.  That is, the spectrum of the matrices $A_n$ are asymptotically dense in $[0,1]$.
\end{cor}
\begin{proof}
Since  $\lim_{n\to\infty} dist(A_n, Nil(\mathcal{M}_n(\mathbb{C}))) = 0$ by Lemma \ref{mas}, Lemma \ref{distrestr} implies $\lim_{n\to\infty} dist(\sigma(A_n), 0) = 0$ and the distance between adjacent eigenvalues (when arranged in increasing order) of each $A_n$ tends to zero as $n$ tends to infinity.  Hence the result easily follows.
\end{proof}
We will require the use of the following trivial result in the proof of Theorem \ref{uspimain}.
\begin{lem}
\label{densespectra2}
Let $\mathfrak{A}$ be a C$^*$-algebra, let $N \in Nor(\mathfrak{A})$, let $(N_n)_{n\geq1}$ be a sequence of normal operators of $\mathfrak{A}$ such that $N = \lim_{n\to\infty} N_n$, and let $U$ be an open subset of $\mathbb{C}$ such that $U \cap \sigma(N) \neq \emptyset$.  Then there exists an $k \in \mathbb{N}$ such that $\sigma(N_n) \cap U \neq \emptyset$ for all $n \geq k$.
\end{lem}
\begin{proof}
Fix $\lambda \in U \cap \sigma(N)$.  By Urysohn's Lemma there exists a continuous function $f$ on $\mathbb{C}$ such that $f|_{U^c} = 0$ yet $f(\lambda) =1$.  Note $f(N) = \lim_{n\to\infty} f(N_n)$ by standard functional calculus results.  If $\sigma(N_n) \cap U  = \emptyset$ for infinitely many $n$, then $f(N_n) =0$ for infinitely many $n$ yet $f(N) \neq 0$ by construction.  This is clearly a contradiction.
\end{proof}
Now we will prove Theorem \ref{uspimain} for positive operators.  Although Proposition \ref{uspipos} is not required in the proof of Theorem \ref{uspimain}, the proof of Proposition \ref{uspipos} contains all the conceptual difficulties and technical approximations thus easing in the comprehension of Theorem \ref{uspimain}.
\begin{prop}
\label{uspipos}
Let $\mathfrak{A}$ be a unital, simple, purely infinite C$^*$-algebra and let $A \in \mathfrak{A}_{+}$.  Then the following are equivalent:
\begin{enumerate}
	\item $A \in \overline{Nil(\mathfrak{A})}$.
	\item $A \in \overline{QuasiNil(\mathfrak{A})}$.
	\item The spectrum of $A$ is connected and contains zero.
\end{enumerate}
\end{prop}
\begin{proof}
Clearly (1) implies (2) and (2) implies (3) is trivial by Lemma \ref{cac0}.  We shall demonstrate that (3) implies (1).
\par
Suppose the spectrum of $A$ is connected and contains zero and let $\epsilon > 0$.  Since $\mathfrak{A}$ is a unital, simple, purely infinite C$^*$-algebra, $\mathfrak{A}$ has real rank zero (see {\cite{Zh}} or Theorem V.7.4 in \cite{Da}).  Thus there exists scalars $0 = a_n < a_{n-1} < \ldots < a_1 = \left\|A\right\|$ and non-zero pairwise orthogonal projections $P^{(1)}_1, \ldots, P^{(1)}_n \in \mathfrak{A}$ such that $\left\|A - A_1\right\| \leq \epsilon$ where 
\[
A_1 := \sum^n_{k=1} a_k P^{(1)}_k.
\]
Moreover, since the spectrum of $A$ is connected, we may also assume that 
\[
\max_{1 \leq k \leq n-1} |a_{k+1} - a_{k}| \leq \epsilon
\]
by Lemma \ref{densespectra2}.  The idea behind the remainder of the proof is to systematically remove the largest eigenvalue of $A_1$ by approximating with a nilpotent operator.  
\par
By Lemma \ref{mas} there exists an $\ell \in \mathbb{N}$, a positive matrix $T_1 \in \mathcal{M}_{\ell}(\mathbb{C})$ with $\left\|T_1\right\| = a_1$, and a nilpotent matrix $M_1 \in \mathcal{M}_{\ell}(\mathbb{C})$ such that $\left\|T_1 - M_1\right\| \leq \epsilon$.  In addition, by a small perturbation, we may assume that the geometric multiplicity of the eigenvalue $a_1$ of $T_1$ is one.  For each $k \in \{2, \ldots, n\}$ let 
\[
\left\{\lambda_{1,k},\lambda_{2,k}, \ldots, \lambda_{m^{(1)}_{k},k}\right\}
\]
be the spectrum of $\sigma(T_1)$ contained in $[a_{k}, a_{k-1})$ counting multiplicity (where zero intersection is possible).  Since $\mathfrak{A}$ is a unital, simple, purely infinite C$^*$-algebra, for each $k \in \{2, \ldots, n\}$ there exists pairwise orthogonal projections 
\[
Q^{(1)}_{1,k}, Q^{(1)}_{2,k}, \ldots, Q^{(1)}_{m^{(1)}_{k}, k}
\]
such that $P^{(1)}_1$ is equivalent $Q^{(1)}_{j,k}$ for each $j \in \left\{1, \ldots, m^{(1)}_k\right\}$ and $\sum^{m_{k}}_{j=1} Q^{(1)}_{j,k} < P^{(1)}_{k}$.
\par
For $k \in \{2, \ldots, n\}$ let 
\[
P^{(2)}_k := P^{(1)}_k - \sum^{m^{(1)}_k}_{j=1} Q^{(1)}_{j,k} > 0
\]
(where the empty sum is the zero projection).  Therefore, if
\[
A'_1 := a_1 P^{(1)}_1 + \sum^n_{k=2} a_k \left(\sum^{m^{(1)}_k}_{j=1} Q^{(1)}_{j,k} \right) \,\,\,\,\,\,\,\,\,\,\mbox{ and }\,\,\,\,\,\,\,\,\,
A_2 := \sum^n_{k=2} a_k P^{(2)}_k
\]
then $A'_1$ and $A_2$ are self-adjoint operators such that $A_1 = A'_1 + A_2$.  Notice if $P^{(2)} := \sum^n_{k=2} P^{(2)}_k$ then $P^{(2)}$ is a non-trivial projection such that 
\[
A_2 \in P^{(2)}\mathfrak{A}P^{(2)}\,\,\,\,\,\,\,\,\,\,\mbox{ and }\,\,\,\,\,\,\,\,\,A'_1 \in \left(I_\mathfrak{A} - P^{(2)}\right)\mathfrak{A}\left(I_\mathfrak{A} - P^{(2)}\right).
\]
Thus the proof will be complete if we can demonstrate that $A'_1$ is within $2\epsilon$ of a nilpotent operator from $\left(I_\mathfrak{A} - P^{(2)}\right)\mathfrak{A}\left(I_\mathfrak{A} - P^{(2)}\right)$ and $A_2$ is within $2\epsilon$ of a nilpotent operator from $P^{(2)}\mathfrak{A}P^{(2)}$.
\par
Recall $P^{(2)}\mathfrak{A}P^{(2)}$ and $\left(I_\mathfrak{A} - P^{(2)}\right)\mathfrak{A}\left(I_\mathfrak{A} - P^{(2)}\right)$ are unital, simple, purely infinite C$^*$-algebras. Moreover, if
\[
A''_1 := a_1P^{(1)}_1 + \sum^n_{k=2} \sum^{m^{(1)}_k}_{j=1} \lambda_{j,k} Q^{(1)}_{j,k} \in \left(I_\mathfrak{A} - P^{(2)}\right)\mathfrak{A}\left(I_\mathfrak{A} - P^{(2)}\right)
\]
then $\left\|A''_1 - A'_1\right\| \leq \epsilon$ by the assumption that $\max_{1 \leq k \leq n-1} |a_{k+1} - a_k| \leq \epsilon$.  Since 
\[
\left\{P^{(1)}_1\right\} \cup \left\{\left\{Q^{(1)}_{j,k}\right\}^{m^{(1)}_k}_{j=1}\right\}^n_{k=2}
\]
are pairwise orthogonal, equivalent projections in $\left(I_\mathfrak{A} - P^{(2)}\right)\mathfrak{A}\left(I_\mathfrak{A} - P^{(2)}\right)$, we can use the partial isometries implementing the equivalence to construct a matrix algebra with these projections as the orthogonal minimal projections.  Moreover, by construction, inside this matrix algebra $A''_1$ has the same spectrum as $T_1$ (including multiplicity) so $A''_1$ can be approximated with the analog of $M_1$ inside $\left(I_\mathfrak{A} - P^{(2)}\right)\mathfrak{A}\left(I_\mathfrak{A} - P^{(2)}\right)$.  Hence $A'_1$ is within $2\epsilon$ of a nilpotent operator from $\left(I_\mathfrak{A} - P^{(2)}\right)\mathfrak{A}\left(I_\mathfrak{A} - P^{(2)}\right)$.
\par
To approximate $A_2$ with a nilpotent operator from $P^{(2)}\mathfrak{A}P^{(2)}$, we repeat the same argument with a positive matrix $T_2$ of norm $a_2$.  Due to the nature of the above approximations, the above process gives a non-trivial projection $P^{(3)} < P^{(2)}$ and a positive operator $A_3$ of $P^{(3)}\mathfrak{A}P^{(3)}$ with spectrum $\{a_3, a_4, \ldots, a_n\}$ such that $A_2 - A_3 \in \left(P^{(2)} - P^{(3)}\right)\mathfrak{A}\left(P^{(2)} - P^{(3)}\right)$ can be approximated within $2\epsilon$ of a nilpotent operator from $\left(P^{(2)} - P^{(3)}\right)\mathfrak{A}\left(P^{(2)} - P^{(3)}\right)$.  By repeating this process a finite number of times (eventually ending with a zero operator), we can write $A_1$ as a finite direct sum of positive matrices each within $2\epsilon$ of a nilpotent matrix from the respective matrix algebra.  Hence $A_1$ is within $2\epsilon$ of a nilpotent operator from $\mathfrak{A}$ and thus $A$ is within $3\epsilon$ of a nilpotent operator from $\mathfrak{A}$.
\end{proof}
In order to adapt the proof of Proposition \ref{uspipos} to normal operators, it is necessary to be able to approximate said operators with normal operators with finite spectra.  This difficult work has already been completed by Lin.
\begin{thm}[Theorem 4.4 in \cite{Li2}]
\label{linfa}
Let $\mathfrak{A}$ be a unital, simple, purely infinite C$^*$-algebra and let $N \in Nor(\mathfrak{A})$.  Then $N$ can be approximated by normal operators with finite spectra if and only if $\lambda I_\mathfrak{A} - N \in \mathfrak{A}_0^{-1}$ for all $\lambda\in \mathbb{C} \setminus \sigma(N)$. 
\end{thm}
It turns out that the condition `$\lambda I_\mathfrak{A} - N \in \mathfrak{A}_0^{-1}$ for all $\lambda\in \mathbb{C} \setminus \sigma(N)$' is a necessary condition for an operator to be a limit of nilpotent operators.
\begin{lem}
\label{concomofid}
Let $\mathfrak{A}$ be a unital C$^*$-algebra and let $T \in \overline{QuasiNil(\mathfrak{A})}$.  Then $\lambda I_\mathfrak{A} - T \in \mathfrak{A}^{-1}_0$ for all $\lambda\in \mathbb{C} \setminus \sigma(T)$.
\end{lem}
\begin{proof}
If $M \in QuasiNil(\mathfrak{A})$ then $\lambda I_\mathfrak{A} - tM$ is invertible for all $\lambda \in \mathbb{C}\setminus \{0\}$ and for all $t \in \mathbb{C}$.  Therefore $\lambda I_\mathfrak{A} - M \in \mathfrak{A}^{-1}_0$ for all $\lambda \in \mathbb{C} \setminus\{0\}$. 
\par
If $T \in \overline{QuasiNil(\mathfrak{A})}$ then $0 \in \sigma(T)$ by Lemma \ref{cac0}.  As $\mathfrak{A}^{-1}_0$ is closed in the relative topology on $\mathfrak{A}^{-1}$, $\lambda I_\mathfrak{A} - T \in \mathfrak{A}_0^{-1}$ for all $\lambda\in \mathbb{C} \setminus \sigma(T)$.  
\end{proof}
With Lemma \ref{concomofid} giving another necessary condition for a normal operator to be a limit of nilpotent operators, we can now address our main theorem.
\begin{thm}
\label{uspimain}
Let $\mathfrak{A}$ be a unital, simple, purely infinite C$^*$-algebra and let $N \in Nor(\mathfrak{A})$.  Then the following are equivalent:
\begin{enumerate}
	\item $N \in \overline{Nil(\mathfrak{A})}$.
	\item $N \in \overline{QuasiNil(\mathfrak{A})}$.
	\item $0 \in \sigma(N)$, $\sigma(N)$ is connected, and $\lambda I_\mathfrak{A} - N \in \mathfrak{A}_0^{-1}$ for all $\lambda\in \mathbb{C} \setminus \sigma(N)$.
\end{enumerate}
\end{thm}
\begin{proof}
Clearly (1) implies (2) and (2) implies (3) is trivial by Lemma \ref{cac0} and Lemma \ref{concomofid}.  We shall demonstrate that (3) implies (1).  As the approximations contained in the proof are identical to those used in Proposition \ref{uspipos}, we will only outline the main technique and omit the approximations.
\par
Suppose $0 \in \sigma(N)$, $\sigma(N)$ is connected, and $\lambda I_\mathfrak{A} - N \in \mathfrak{A}_0^{-1}$ for all $\lambda\in \mathbb{C} \setminus \sigma(N)$.  Fix $\epsilon > 0$ and for each $(n,m) \in \mathbb{Z}^2$ let 
\[
B_{n,m} := \left(\epsilon n - \frac{\epsilon}{2}, \epsilon n + \frac{\epsilon}{2} \right]  + \mathrm{i} \left( \epsilon m - \frac{\epsilon}{2}, \epsilon m + \frac{\epsilon}{2} \right] \subseteq \mathbb{C}.
\]
\par
By Theorem \ref{linfa} there exists a normal operator $N_\epsilon$ with finite spectrum such that $\left\|N - N_\epsilon \right\| \leq \epsilon$.  For each $(n,m) \in \mathbb{Z}^2$ we label the box $B_{n,m}$ relevant if $\sigma(N_\epsilon) \cap B_{n,m} \neq \emptyset$ and we label the box $B_{n,m}$ irrelevant if $\sigma(N_\epsilon) \cap B_{n,m} = \emptyset$.  Since $\sigma(N)$ is connected, we may assume (via Lemma \ref{densespectra2}) that the union of all relevant $B_{n,m}$ is a connected set and $B_{0,0}$ is relevant.  By a perturbation of at most $\epsilon$, we can assume that $\sigma(N_\epsilon)$ is precisely the centres of all relevant boxes and $\left\|N - N_\epsilon\right\| \leq 2\epsilon$.
\par
The remainder of the proof is similar in nature to the proof of Proposition \ref{uspipos} in that we will use a recursive algorithm to write $N_\epsilon$ as a finite direct sum of matrices inside of $\mathfrak{A}$ each of which is within $5\epsilon$ of the set of nilpotent matrices.  If the only relevant box is $B_{0,0}$, the algorithm may stop as $N_\epsilon$ is the zero operator and thus nilpotent.  Otherwise we label a relevant box bad if its removal disconnects the union of the relevant boxes or it is $B_{0,0}$ and we label a relevant box good if it is not bad.  Elementary graph theory implies that at least one box is good.
\par
Let $B_{n_0, m_0}$ be a good, relevant box.  Since the union of the relevant boxes is connected, there exists a continuous path $\gamma : [0,1] \to \mathbb{C}$ that connects 0 to $\epsilon n_0 + \mathrm{i}\epsilon m_0$ whose image lies in the union of the relevant boxes.  By Lemma \ref{mas} and since $\gamma$ can be approximated uniformly by a polynomial that vanishes at zero, there exists an $\ell \in \mathbb{N}$, a normal operator $N_\ell \in \mathcal{M}_\ell(\mathbb{C})$, and a nilpotent $M_\ell \in \mathcal{M}_\ell(\mathbb{C})$ such that the spectrum of $N_\ell$ is contained within an $\epsilon$-neighbourhood of the union of relevant boxes and $\left\|N_\ell - M_\ell\right\| \leq \epsilon$.  By perturbing the eigenvalues of $N_\ell$ by at most $4\epsilon$, we can assume that the spectrum of $N_\ell$ is precisely a subset of the centres of relevant boxes, the multiplicity of $\epsilon n_0 +\mathrm{i} \epsilon m_0$ is precisely one, and $\left\|M_\ell - N_\ell\right\| \leq 5\epsilon$.
\par
For each $(n,m) \in\mathbb{Z}^2$ let $P_{n,m}$ be the spectral projection of $N_\epsilon$ for the box $B_{n,m}$.  Using $P_{n_0, m_0}$ as a main projection, for each other $(n,m) \in \mathbb{Z}^2$ such that $\epsilon n +\mathrm{i}\epsilon m$ is in the spectrum of $N_\ell$ we can find the algebraic multiplicity of the eigenvalue $\epsilon n +\mathrm{i}\epsilon m$ of $N_\ell$ many orthogonal subprojections of $P_{n,m}$ whose sum is strictly less then $P_{n,m}$ and each of which is equivalent to $P_{n_0, m_0}$.  Thus, as in the proof of Proposition \ref{uspipos}, we can find a projection $P_1 \in \mathfrak{A}$ such that $P_1$ commutes with $N_\epsilon$, $P_1N_\epsilon P_1$ can be approximated by a nilpotent operator from $P_1\mathfrak{A}P_1$ within $5\epsilon$, and $(I-P_1)N_\epsilon(I-P_1)$ has the same spectrum as $N_\epsilon$ minus $\epsilon n_0 + \mathrm{i}\epsilon m_0$.
\par
By our selection of $(n_0, m_0)$ and choice of projection $P_1$, the number of relevant $B_{n,m}$ for $(I-P_1)N_\epsilon(I-P_1)$ is one less than the number of relevant $B_{n,m}$ for $N_\epsilon$ and the union of the relevant $B_{n,m}$ for $(I-P_1)N_\epsilon(I-P_1)$ is connected and contains $B_{0,0}$.  Thus, by repeating the above process a finite number of times, we obtain a nilpotent operator $M \in \mathfrak{A}$ such that $\left\|N - M\right\| \leq 7\epsilon$.  Hence the result follows. 
\end{proof}
In the case of our C$^*$-algebra is not a purely infinite C$^*$-algebra, we note that the following can easily be proved using the techniques illustrated above.
\begin{lem}
Let $\mathfrak{A}$ be a unital, simple C$^*$-algebra and let $N \in Nor(\mathfrak{A})$ be such that $\sigma(N)$ is connected and contains zero.  If $N = \lim_{n\to\infty} \sum^{m_n}_{k=1} a_{k,n} P_{k,n}$ where $a_{k,n} \in \mathbb{C}$ and $P_{k,n}$ are infinite projections with $\sum^{m_n}_{k=1} P_{k,n} = I_\mathfrak{A}$ then $N \in \overline{Nil(\mathfrak{A})}$.
\end{lem}
\begin{proof}
The conditions that $\mathfrak{A}$ is simple and the projections are infinite imply that the projections are properly infinite (see Theorem V.5.1 in \cite{Da}) and every projection is equivalent to a subprojection of any infinite projection (see Lemma V.5.4 in \cite{Da}).  Thus the process used above works (where we note the small technical detail that, when removing one projection from the sum, we can still take the differences containing the other projections to be infinite by showing that they containing a strict subprojection equivalent to the original $P_{k,n}$ by Theorem V.5.1 in \cite{Da}).  
\end{proof}
With the proof of Theorem \ref{uspimain} complete, we turn our attention to other interesting questions pertaining to limits of nilpotent operators in unital, simple, purely infinite C$^*$-algebras.  To begin, we recall that Corollary 6 in \cite{He2} shows that the closure of $Nil(\mathcal{B}(\mathcal{H}))+Nil(\mathcal{B}(\mathcal{H}))$ contained every normal operator.  We now demonstrate a similar result for unital, simple, purely infinite C$^*$-algebras.
\begin{thm}
\label{uspisum}
Let $\mathfrak{A}$ be a unital, simple, purely infinite C$^*$-algebra.  Then 
\[
\mathfrak{A}_{sa} \subseteq \overline{\{M_1 + M_2 \, \mid \, M_1, M_2 \in Nil(\mathfrak{A})\}}
\]
and 
\[
\mathfrak{A} \subseteq \overline{\{M_1 + M_2 + M_3 + M_4\, \mid \, M_1, M_2, M_3, M_4 \in Nil(\mathfrak{A})\}}.
\]
\end{thm}
\begin{proof}
Clearly the second result follows from the first by considering real and imaginary parts.  To prove the first result, we will first demonstrate that
\[
I_\mathfrak{A} \in \overline{\{M_1 + M_2 \, \mid \, M_1, M_2 \in Nil(\mathfrak{A})\}}.
\]
Note that there exists a positive operator $A \in \mathfrak{A}$ such that $\sigma(A) = [0,1]$.  Thus $A$ and $I_\mathfrak{A} - A$ are limits of nilpotent operators by Theorem \ref{uspimain} (or Proposition \ref{uspipos}) which completes the claim.
\par
Let $A \in \mathfrak{A}_{sa}$ be arbitrary and fix $\epsilon > 0$.  Since $\mathfrak{A}$ has real rank zero (see Theorem V.7.4 in \cite{Da}), there exists non-zero pairwise orthogonal projections $\{P_k\}^n_{k=1} \subseteq \mathfrak{A}$ and scalars $\{a_k\}^n_{k=1}$ such that $\left\| \sum^n_{k=1} a_k P_k- A\right\| < \epsilon$.  Since each $P_k \mathfrak{A}P_k$ is a unital, simple, purely infinite C$^*$-algebra with unit $P_k$, $P_k$ is a limit of the sum of two nilpotent operators from $P_k \mathfrak{A}P_k$.  Since the finite direct sum of nilpotent operators is a nilpotent operator, $\sum^n_{k=1} a_k P_k$ is a limit of sums of two nilpotent operators from $\mathfrak{A}$ and thus the result follows.  
\end{proof}
\begin{cor}
Let $\mathfrak{A}$ be a unital, simple, purely infinite C$^*$-algebra and let $N \in Nor(\mathfrak{A})$ be such that $\lambda I_\mathfrak{A} - N \in \mathfrak{A}^{-1}_0$ for all $\lambda\in \mathbb{C}\setminus \sigma(N)$.  Then
\[
N \in \overline{\{M_1 + M_2 \, \mid \, M_1, M_2 \in Nil(\mathfrak{A})\}}
\]
\end{cor}
\begin{proof}
The result follows from the same argument in Theorem \ref{uspisum} where $N$ can be approximated by normal operators with finite spectrum by Theorem \ref{linfa}.
\end{proof}
We note that if $\mathfrak{A} := \mathcal{O}_n$ is the Cuntz algebra generated by $n$ isometries then $\mathfrak{A}^{-1}_0 = \mathfrak{A}^{-1}$ by \cite{Cu}.  Thus $Nor(\mathcal{O}_n) \subseteq \overline{\{M_1 + M_2 \, \mid \, M_1, M_2 \in Nil(\mathcal{O}_n)\}}$ for all $n \in \mathbb{N}$.
\par
In Corollary 9 of \cite{He2}, Herrero determined the distance from a fixed projection in $\mathcal{B}(\mathcal{H})$ to the nilpotent and quasinilpotent operators was either 0, 1, or $\frac{1}{2}$ and gave necessary and sufficient conditions for each distance.  Using the structure of projections in unital, simple, purely infinite C$^*$-algebras, it is possible to imitate Herrero's work.  We begin by noting the following trivial result which follows from the fact that any element of the open unit ball around the identity in a C$^*$-algebra is invertible and by Lemma \ref{distrestr}.
\begin{lem}[Lemma 8.1 in \cite{Sk}]
\label{projlemma}
Let $\mathfrak{A}$ be a unital C$^*$-algebra.  Then 
\[
dist(I_\mathfrak{A}, Nil(\mathfrak{A})) = dist(I_\mathfrak{A}, QuasiNil(\mathfrak{A}))= 1.
\]
Moreover, if $P \in \mathfrak{A}$ is a non-trivial projection then
\[
\frac{1}{2} \leq dist(P, QuasNil(\mathfrak{A})) \leq dist(P, Nil(\mathfrak{A})) \leq 1.
\]
\end{lem}
\begin{thm}
\label{uspipro}
Let $\mathfrak{A}$ be a unital, simple, purely infinite C$^*$-algebra and let $P \in
\mathfrak{A}$ be a projection.  Then
\begin{enumerate}
        \item $dist(P, Nil(\mathfrak{A}))= dist(P, QuasiNil(\mathfrak{A})) = 0$ if $P = 0$,
        \item $dist(P, Nil(\mathfrak{A}))= dist(P, QuasiNil(\mathfrak{A})) = 1$ if $P = I_\mathfrak{A}$, and
        \item $dist(P, Nil(\mathfrak{A}))= dist(P, QuasiNil(\mathfrak{A})) = \frac{1}{2}$ otherwise.
\end{enumerate}
\end{thm}
\begin{proof}
Clearly (1) and (2) hold by Lemma \ref{projlemma}.  To see that (3) holds,  it suffices to show 
\[
dist(P, Nil(\mathfrak{A})) \leq \frac{1}{2}
\]
by Lemma \ref{projlemma}.  Since $I_\mathfrak{A} - P$ is a properly infinite projection, for each $k \in \mathbb{N}$ there exists pairwise orthogonal projections
$Q_{1,k}, Q_{2,k}, \ldots, Q_{k, k}$ such that $P$ is equivalent to $Q_{j,k}$ for each $j \in \{1, \ldots, k\}$ and $\sum^{k}_{j=1} Q_{j,k} < I_\mathfrak{A}-P$. 
\par
Let
\[
Q_k := P + \sum^k_{j=1} Q_{j,k}.
\]
Then $P \in Q_k\mathfrak{A}Q_k$ for all $k \in
\mathbb{N}$.  Thus it suffices to show that
\[
\inf_{k\geq 1} dist(P, Nil(Q_k\mathfrak{A}Q_k))
\leq \frac{1}{2}.
\]
Since
\[
\{P\} \cup \{Q_{j,k}\}^k_{j=1}
\]
is a set of equivalent, pairwise orthogonal projections in $\mathfrak{A}$, we can use the
partial isometries implementing the equivalence to construct a copy of
$\mathcal{M}_{k+1}(\mathbb{C})$ with these projections as the orthogonal minimal
projections.  Moreover, by construction, inside this matrix algebra $P$ is a rank one
projection.  Thus, by Theorem 2.12 in \cite{He1}, $P$ is within $\frac{1}{2} +
\sin\left(\frac{\pi}{m_k+1}\right)$ (where $m_k$ is the integer part of $\frac{k}{2}$) of
a nilpotent matrix.  Thus
\[
dist(P, Nil(Q_k\mathfrak{A}Q_k)) \leq \frac{1}{2}
+ \sin\left(\frac{\pi}{m_k+1}\right)
\]
so the result follows.
\end{proof}

\section{AF C$^*$-Algebras}
\label{sec:AFALGEBRAS}

In this section we will investigate when a normal operator in a AF C$^*$-algebra is a norm limit of nilpotent operators.  The study of such operators is intrinsically related to how normal matrices can be asymptotically approximated by nilpotent matrices as we allow the dimension of our matrices to increase.  Proposition \ref{boundedAFD} will provide conditions on an AF C$^*$-algebra that guarantee the intersection of the normal operators and the quasinilpotent operators is trivial whereas Theorem \ref{AFpos} exhibits an AF C$^*$-algebra $\mathfrak{A}$ where $\mathfrak{A}_{sa} \cap \overline{Nil(\mathfrak{A})} \neq \{0\}$.  Moreover, in Theorem \ref{UHFnil} which is the main result of this section, we will demonstrate that every UHF C$^*$-algebra has a normal operator with spectrum equal to the closed unit disk that is a norm limit of nilpotent operators. All of this together (along with Proposition \ref{AFwithdensenil}) implies that the study of $Nor(\mathfrak{A}) \cap \overline{Nil(\mathfrak{A})}$ for AF C$^*$-algebras $\mathfrak{A}$ is incredibly complex.
\par
We begin with the following important result.
\begin{prop}
\label{naqniAF}
Let $\mathfrak{A}$ be an AF C$^*$-algebra and write $\mathfrak{A} = \overline{\bigcup_{k \geq 1} \mathfrak{A}_k}$ where each $\mathfrak{A}_k$ is a finite dimensional C$^*$-algebra.  For each $T \in \mathfrak{A}$ following are equivalent:
\begin{enumerate}
	\item $T \in \overline{QuasiNil(\mathfrak{A})}$.
	\item $T \in \overline{Nil(\mathfrak{A})}$.
	\item $T \in \overline{\bigcup_{k \geq 1}Nil(\mathfrak{A}_k)}$.
\end{enumerate}
\end{prop}
\begin{proof}
Clearly (3) implies (2) and (2) implies (1).  Suppose $T \in \overline{QuasiNil(\mathfrak{A})}$.  Let $\epsilon > 0$ and choose $M \in QuasiNil(\mathfrak{A})$ such that $\left\|T - M\right\| < \epsilon$.  Since $M \in \overline{\bigcup_{k \geq 1} \mathfrak{A}_k}$ and by the semicontinuity of the spectrum, there exist an $k \in \mathbb{N}$ and an operator $M_0 \in \mathfrak{A}_k$ such that $\left\|M_0 - M\right\| < \epsilon$ and 
\[
\sigma(M_0) \subseteq \left\{z \in \mathbb{C} \, \mid \, dist(z, \sigma(M)) < \epsilon\right\} = \left\{z \in \mathbb{C} \, \mid \, |z| < \epsilon\right\}.
\]
Since $\mathfrak{A}_k$ is a finite dimensional C$^*$-algebra, $\mathfrak{A}_k$ is a direct sum of matrix algebras.  Thus $M_0$ is unitarily equivalent to a direct sum of upper triangular matrices.  Each of these upper triangular matrices is the sum of a nilpotent matrix and a diagonal matrix whose diagonal entries are in $\sigma(M_0)$.  Since the equivalence is via a unitary, by subtracting the diagonal part we obtain an $M' \in Nil(\mathfrak{A}_k)$ such that 
\[
\left\|M_0 - M'\right\| \leq \sup\left\{|z| \, \mid \, z \in \sigma(M_0)\right\} < \epsilon.
\]
Therefore $\left\|T - M'\right\| < 3\epsilon$ completing the proof.
\end{proof}
\begin{rem}
\label{AFremarks}
The study of which normal operators of an AF C$^*$-algebra are in the closure of the nilpotent operators is intrinsically connected to the distribution of eigenvalues of normal matrices that are asymptotically approximated by nilpotent matrices as we allow the dimension of the matrices to increase.  
\par
Indeed if $\mathfrak{A}$ is an AF C$^*$-algebra with $\mathfrak{A} = \overline{\bigcup_{k \geq 1} \mathfrak{A}_k}$ where $\mathfrak{A}_1 \stackrel{\alpha_1}{\rightarrow} \mathfrak{A}_2 \stackrel{\alpha_2}{\rightarrow} \mathfrak{A}_3 \stackrel{\alpha_3}{\rightarrow} \cdots$ is a direct limit of finite dimensional C$^*$-algebras with $\alpha_k$ injective for all $k \in \mathbb{N}$, then it is easy to see by Proposition \ref{naqniAF} and by \cite{Li1} that $N \in Nor(\mathfrak{A}) \cap \overline{Nil(\mathfrak{A})}$ if and only if for each $k \in \mathbb{N}$ there exists an $N_k \in Nor(\mathfrak{A}_k)$ such that $N = \lim_{k\to\infty} N_k$ and $\lim_{k\to\infty} dist(N_k, Nil(\mathfrak{A}_k)) = 0$.  Moreover, since $N = \lim_{k\to\infty} N_k$, $\lim_{k\to\infty}\left\|\alpha_k(N_k) - N_{k+1}\right\| = 0$.  This is possible only if for each $k \in \mathbb{N}$ the eigenvalues of $\alpha_k(N_k)$ and  $N_{k+1}$ (including multiplicities) can be paired together in a manner such that the maximum of the absolute values of the differences tends to zero as $k$ tends to infinity.   
\par
Similarly, if $N_k \in Nor(\mathfrak{A}_k)$ can be chosen such that for each $k \in \mathbb{N}$ the eigenvalues of $\alpha_k(N_k)$ and  $N_{k+1}$ (including multiplicities) can be paired together inside the appropriate direct summand of $\mathfrak{A}_{k+1}$ in a manner such that the maximum of the absolute values of the differences tends to zero as $k$ tends to infinity and $\lim_{k\to\infty} dist(N_k, Nil(\mathfrak{A}_k)) = 0$, then, by taking unitary conjugates of the matrices $N_k$, it is possible to construct a Cauchy sequence in $\mathfrak{A}$ that converges to a normal operator $N$ in the closure of the nilpotent operators.
\end{rem}
\begin{exam}
\label{2nnil}
For each $n \in \mathbb{N}$ let $A_n \in \mathcal{M}_{2^n}(\mathbb{C})$ be a diagonal matrix with spectrum $\left\{\frac{1}{2^n}, \frac{2}{2^n}, \ldots, 1\right\}$.  Then 
\[
\liminf_{n\to\infty} dist(A_n, Nil(\mathcal{M}_{2^n}(\mathbb{C}))) > 0.
\]
To see this, we note that the sequence $(A_n)_{n\geq1}$ can be used to construct a Cauchy sequence in the $2^\infty$-UHF C$^*$-algebra $\mathfrak{A}$ that converges to a non-zero positive operator $A$.  If $\liminf_{n\to\infty} dist(A_n, Nil(\mathcal{M}_{2^n}(\mathbb{C}))) = 0$ then $A$ would be the limit of elements of $Nil(\mathfrak{A})$ which would contradict Proposition 4.6 in \cite{Sk} as $\mathfrak{A}$ has a faithful tracial state.  
\par
Alternatively
\[
\liminf_{n\to\infty} dist(A_n, Nil(\mathcal{M}_{2^n}(\mathbb{C}))) \geq \frac{1}{2}
\]
since the normalized trace on $\mathcal{M}_{2^n}(\mathbb{C})$ has norm one, the normalized traces of $A_n$ tend to $\frac{1}{2}$ as $n$ tends to infinity, and the trace of any nilpotent matrix is zero.
\end{exam}
Note, in the above example, we can view each $A_n$ as a positive operator whose spectrum is the first $2^n$ entries of the sequence $\{1, \frac{1}{2}, \frac{3}{4}, \frac{1}{4}, \frac{7}{8}, \ldots\}$.  Thus, by Remarks \ref{AFremarks}, we are interested in the following question: ``Given a sequence $(a_n)_{n\geq 1} \in \ell_{\infty}(\mathbb{N})$ does $\liminf_{n\to\infty} dist(diag(a_1, \ldots, a_n), Nil(\mathcal{M}_n(\mathbb{C}))) = 0$?"  An application of Lemma \ref{distrestr} implies $\overline{\{a_n\}_{n\geq1}}$ must be a connected set containing zero in order for an affirmative answer to this question.  Thus the following is of particular interest.
\begin{prop}
There exists a sequence $(a_n)_{n\geq 1} \in \ell_\infty(\mathbb{C})_+$ with $\overline{\{a_n\}_{n\geq1}} = [0,1]$ such that 
\[
\liminf_{n\to\infty} dist(diag(a_1, \ldots, a_n), Nil(\mathcal{M}_n(\mathbb{C}))) = 0.
\]
\end{prop}
\begin{proof}
By Lemma \ref{mas} for each $n \in \mathbb{N}$ there exists a positive matrix $A_n \in \mathcal{M}_{n}(\mathbb{C})$ of norm one such that $\lim_{n\to\infty}dist(A_n, Nil(\mathcal{M}_n(\mathbb{C}))) = 0$.  Choose $n_1 \in \mathbb{N}$ such that $dist(A_{n_1}, Nil(\mathcal{M}_{n_1}(\mathbb{C}))) \leq 1$.  Let the first $n_1$ of the scalars $a_j$ be the eigenvalues of $A_{n_1}$ (including multiplicity).
\par
Let $R_1 := A_{n_1}$.  By Corollary \ref{densespectra} $\sigma(A_n)$ progressively gets dense in $[0,1]$ as $n$ increases.  Therefore there exists an $n_2 \in \mathbb{N}$ such that 
\[
\sigma(A_{n_2}) \cap \left[ \frac{k}{2^2}, \frac{k+1}{2^2}\right) \neq \emptyset
\]
for all $k \in \{0,1, 2, 3\}$ and 
\[
dist(A_{n_2}, Nil(\mathcal{M}_{n_2}(\mathbb{C}))) \leq \frac{1}{2^2}.
\]
By comparing the eigenvalues of $R_1$ and $A_{n_2}$ there exists an $m_1 \in \mathbb{N}$ and an injective map  $f_1$ from the eigenvalues of $R_1$ (including multiplicity) to the eigenvalues of $A_{n_2}^{\oplus m_1}$ (including multiplicity) such that $|\lambda - f_1(\lambda)| \leq\frac{1}{4}$ for all eigenvalues $\lambda$ of $R_1$ (including multiplicity).  Therefore, if $A_{n_2}^{\oplus m_1} \ominus R_1$ denotes the $(m_1n_2-n_1)\times (m_1n_2-n_1)$ diagonal matrix whose diagonal entries are the eigenvalues of $A_{n_2}^{\oplus m_1}$ (including multiplicities) excluding $f_1(\lambda)$ for all eigenvalues $\lambda$ of $R_1$ (including multiplicity), then 
\[
R_2 := R_1 \oplus (A_{n_2}^{\oplus m_1} \ominus R_1)
\]
is within $\frac{1}{4}$ of a unitary conjugate of $A_{n_2}^{\oplus m_1}$ and thus 
\[
\begin{array}{rcl}
dist(R_2, Nil(\mathcal{M}_{n_2m_1}(\mathbb{C})) &\leq& \frac{1}{4} + dist(A_{n_2}^{\oplus m_1}, Nil(\mathcal{M}_{n_2m_1}(\mathbb{C}))  \\
 &\leq& \frac{1}{4} + dist(A_{n_2}, Nil(\mathcal{M}_{n_2}(\mathbb{C})) \leq \frac{1}{2}. 
 \end{array} 
\]
Thus define the next $m_1n_2-n_1$ of the scalars $a_j$ to be the eigenvalues of $A_{n_2}^{\oplus m_1} \ominus R_1$ (including multiplicity).
\par
By continuing this process ad infinitum, the desired sequence $(a_n)_{n\geq 1}$ is obtained.
\end{proof}
Of course the existence of the above sequence does not imply that there exists an AF C$^*$-algebra with a non-zero positive operator in the closure of the nilpotent operators as the structure required for such an operator is more complex (see Remarks \ref{AFremarks}).  However, an example of such a AF C$^*$-algebra is an easy application of the theory developed in Section \ref{sec:USPI}.
\begin{thm}
\label{AFpos}
There exists an AF C$^*$-algebra $\mathfrak{A}$ such that $\mathfrak{A}_{+} \cap \overline{Nil(\mathfrak{A})} \neq \{0\}$.
\end{thm}
\begin{proof}
Let $\mathcal{O}_2$ be the Cuntz algebra generated by two isometries.  Since $\mathcal{O}_2$ is a separable, nuclear C$^*$-algebra, the cone of $\mathcal{O}_2$, $\mathfrak{C} := C_0((0,1], \mathcal{O}_2)$, is AF-embeddable (see Proposition 2 in \cite{Oz} or Theorem 8.3.5 in \cite{BO}).  Hence there exists an AF C$^*$-algebra $\mathfrak{A}$ such that $\mathfrak{C} \subseteq \mathfrak{A}$.  Thus it suffices to show $\mathfrak{C}_{+} \cap \overline{Nil(\mathfrak{C})} \neq \{0\}$.
\par
Let $A \in (\mathcal{O}_2)_+ \setminus \{0\}$ be such that $\sigma(A) = [0,1]$ and let $A' \in \mathfrak{C}_+$ be defined by $A'(x) = A x$ for all $x \in (0,1]$.  Since $A \neq 0$, $A' \neq 0$.  Since $A \in \overline{Nil(\mathcal{O}_2)}$ by Theorem \ref{uspimain} (or simply Proposition \ref{uspipos}), it is trivial to verify that $A' \in \overline{Nil(\mathfrak{C})}$ as desired.
\end{proof}
Using Theorem \ref{AFpos} and Proposition \ref{naqniAF}, it is easy to obtain the following that enables us to improve Lemma \ref{mas} by bounding the nilpotency degrees of the approximating nilpotent matrices.  Theorem \ref{AFpos}, Proposition \ref{naqniAF}, and Remarks \ref{AFremarks} together also imply that Lemma \ref{mas} holds with the additional property that the distribution of eigenvalues of the sequence $A_n$ is `not too poorly behaved'.
\begin{cor}
There exists an increasing sequence of natural numbers $(k_n)_{n\geq1}$ and a sequence of positive matrices $A_n \in \mathcal{M}_{k_n}(\mathbb{C})$ of norm one such that for every $\epsilon > 0$ there exists an index $m \in \mathbb{N}$ and a $\ell \in \mathbb{N}$ such that 
\[
dist(A_n, Nil_\ell(\mathcal{M}_{k_n}(\mathbb{C}))) < \epsilon
\]
for all $n\geq m$ (where $Nil_\ell(\mathcal{M}_{k_n}(\mathbb{C}))$ is the set of nilpotent $k_n \times k_n$-matrices of nilpotency index at most $\ell$).
\end{cor}
Next we have the following trivial observation that demonstrates several AF C$^*$-algebras where no non-zero normal operators are limits of quasinilpotent operators.
\begin{prop}
\label{boundedAFD}
Suppose $\mathfrak{A} = \overline{\bigcup_{k \geq 1} \mathfrak{A}_k}$ where $\mathfrak{A}_1 \stackrel{\alpha_1}{\rightarrow} \mathfrak{A}_2 \stackrel{\alpha_2}{\rightarrow} \mathfrak{A}_3 \stackrel{\alpha_3}{\rightarrow} \cdots$ is a direct limit of finite dimensional C$^*$-algebras with $\alpha_k$ injective for all $k \in \mathbb{N}$.  If $\mathfrak{A}_k = \oplus^{m_k}_{j=1} \mathcal{M}_{n_{j,k}}(\mathbb{C})$ and $\{n_{j,k}\}_{j,k\geq 1}$ is a bounded set, then $Nor(\mathfrak{A}) \cap \overline{QuasiNil(\mathfrak{A})} = \{0\}$.
\end{prop}
\begin{proof}
Suppose $N \in Nor(\mathfrak{A}) \cap \overline{QuasiNil(\mathfrak{A})}$ and let $\ell := \sup_{j,k\geq 1} n_{j,k} < \infty$.  Therefore $M^\ell = 0$ for all $M \in \bigcup_{k\geq 1} Nil(\mathfrak{A}_k)$ so $N^\ell = 0$ by Proposition \ref{naqniAF}.  Hence $N = 0$.
\end{proof}
The main result of this section is Theorem \ref{UHFnil} which gives examples of normal operators in each UHF C$^*$-algebra that are limits of nilpotent operators.  This result is slightly surprising since every UHF C$^*$-algebra has a faithful tracial state yet \cite{Sk} demonstrated that faithful tracial states impose restrictions on when normal operators can be limits of nilpotent operators.  In particular, Proposition 4.6 in \cite{Sk} shows that $\mathfrak{A}_{sa} \cap \overline{QuasiNil(\mathfrak{A})} = \{0\}$ for every UHF C$^*$-algebra $\mathfrak{A}$ (also see Corollary 4.8, Lemma 4.12, and Theorem 4.13 in \cite{Sk}).
\par
The main tool in this construction is Lemma 5.4 in \cite{Sk} which is based on Section 2.3.3 of \cite{He1}. 
\begin{lem}[Lemma 5.4 in \cite{Sk}]
\label{nmctnannoy}
Let $n,m \in \mathbb{N}$ with $m\geq 2$ and choose $0 = a_0 < a_1 < a_2 < \ldots < a_m = 1$.  Then there exists an $M \in Nil(\mathcal{M}_{(2m+1)n+1}(\mathbb{C}))$ and an $N \in Nor(\mathcal{M}_{(2m+1)n+1}(\mathbb{C}))$ such that
\[
\left\|M-N\right\| \leq \frac{\pi}{n} + \max_{0 \leq k \leq m-1}  |a_{k+1} - a_k|
\]
and
\[
\sigma(N) =  \left\{a_k \mathrm{e}^{\frac{\pi \mathrm{i}}{n} j} \, \mid \, j \in \{1,\ldots, 2n\}, k \in \{0,\ldots, m\}\right\}
\]
with the multiplicity of zero being $n + 1$ and the multiplicity of every other eigenvalue being one.
\end{lem}
The following result was known to Marcoux and was communicated to the author.
\begin{thm}[Marcoux]
\label{UHFnil}
Let $\mathfrak{A}$ be a nonelementary UHF C$^*$-algebra.  There exists an $N \in Nor(\mathfrak{A}) \cap \overline{Nil(\mathfrak{A})}$ such that $\sigma(N)$ is the closed unit disk.
\end{thm}
\begin{proof}
Write $\mathfrak{A} = \overline{\bigcup_{k \geq 1} \mathcal{M}_{\ell_k}(\mathbb{C})}$ where $\mathcal{M}_{\ell_1}(\mathbb{C}) \stackrel{\alpha_1}{\rightarrow} \mathcal{M}_{\ell_2}(\mathbb{C}) \stackrel{\alpha_2}{\rightarrow} \mathcal{M}_{\ell_3}(\mathbb{C}) \stackrel{\alpha_3}{\rightarrow} \cdots$ is a direct limit of full matrix algebras with $\alpha_k$ injective for all $k \in \mathbb{N}$.  Moreover we can assume that $\frac{\ell_{k+1}}{\ell_k}$ is composite for all $k \in \mathbb{N}$ and $\ell_1 \geq 11$.
\par
For each $k \in \mathbb{N}$ we will construct $n_k, m_k \in \mathbb{N}$ and $q_k \in \mathbb{N} \cup \{0\}$ such that $m_1, n_1 \geq 2$, $(2m_k+1)n_k + 1 + q_k = \ell_k$ for all $k \in \mathbb{N}$, $2m_k \leq m_{k+1}$ for all $k \in \mathbb{N}$, $2n_k \leq n_{k+1}$ for all $k \in \mathbb{N}$, and, if $N_k \in Nor(\mathcal{M}_{(2m_k+1)n_k + 1 + q_k}(\mathbb{C}))$ is a specific unitary conjugate of the normal matrix obtain by taking the direct sum of the $q_k \times q_k$ zero matrix with the normal matrix from Lemma \ref{nmctnannoy} with $n = n_k$, $m = m_k$, and $a_j = \frac{j}{m_k}$ for all $j \in \{0,1,\ldots, m_k\}$ then $(N_k)_{k\geq1}$ is a Cauchy sequence in $\mathfrak{A}$. 
\par
If such a sequence exists then, since $\lim_{k\to \infty} m_k = \infty$ and $\lim_{k\to\infty} n_k = \infty$ and since adding a zero direct summand at most decreases the distance to the nilpotent operators, Lemma \ref{nmctnannoy} implies 
\[
\lim_{k\to \infty} dist(N_k, Nil(\mathcal{M}_{\ell_k}(\mathbb{C}))) = 0.
\]
Thus, if $N = \lim_{k\to\infty} N_k$ then $N \in Nor(\mathfrak{A}) \cap \overline{Nil(\mathfrak{A})}$ by construction.  Since $\left\|N_k\right\| \leq 1$, $\left\|N\right\| \leq 1$.  Since $\lim_{k\to \infty} m_k = \infty$ and $\lim_{k\to\infty} n_k = \infty$, Lemma \ref{nmctnannoy} implies the intersection of $\sigma(N_k)$ with any open subset of the closed unit ball is non-empty for sufficiently large $k$.  This implies $\sigma(N)$ is the closed unit disk by the semicontinuity of the spectrum.
\par
To show that the claim is true, let $m_1 = 2$ and select $n_1 \in \mathbb{N}$ with $n_1 \geq 2$ and $q_1 \in \{0,1,2,3,4\}$ such that $\ell_1 = (2m_1+1)n_1 + 1 + q_1$.  Let $N_1$ be as described above.  
\par
Suppose we have performed the construction for some fixed $k \in \mathbb{N}$.  Since $\frac{\ell_{k+1}}{\ell_k}$ is composite, we may write $\frac{\ell_{k+1}}{\ell_k} = pz$ where $p,z \geq 2$.  Then, when we view $N_k$ as an element of $\mathcal{M}_{\ell_{k+1}}(\mathbb{C})$, each eigenvalue of $N_k$ has $pz$ times the multiplicity it did in $\mathcal{M}_{\ell_k}(\mathbb{C})$.  Let $n_{k+1} := pn_k \geq 2n_k$ and $m_{k+1} := z m_k \geq 2m_k$.  Then 
\[
(2m_{k+1}+1)n_{k+1} + 1 = \ell_{k+1} - ((z-1)pn_k + pz + pzq_k-1).
\]
Thus let $q_{k+1} := ((z-1)pn_k + pz+ pzq_k-1) \geq 0$ so 
\[
(2m_{k+1}+1)n_{k+1} + 1 + q_{k+1} = \ell_{k+1}.
\]
If $N'_{k+1}$ is the normal matrix obtain by taking the direct sum of the $q_k \times q_k$ zero matrix with the normal matrix from Lemma \ref{nmctnannoy} with $n = n_{k+1}$, $m = m_{k+1}$, and $a_j = \frac{j}{m_{k+1}}$ for all $j \in \{0,1,\ldots, m_{k+1}\}$, then, by construction, we can pair the eigenvalues of $N_k$ (including multiplicity) when viewed an element of $\mathcal{M}_{\ell_{k+1}}(\mathbb{C})$ with the eigenvalues of $N'_{k+1}$ in a bijective way such that the difference of any pair is at most $\frac{\pi}{n_k} + \frac{1}{m_k}$ by our knowledge of the eigenvalues from Lemma \ref{nmctnannoy}.  Thus there exists a unitary conjugate $N_{k+1}$ of $N'_{k+1}$ that is within $\frac{\pi}{n_k} + \frac{1}{m_k}$ of the image of $N_k$ in $\mathcal{M}_{\ell_{k+1}}(\mathbb{C})$.  Since $2m_k \leq m_{k+1}$ for all $k \in \mathbb{N}$ and $2n_k \leq n_{k+1}$ for all $k \in \mathbb{N}$, this implies $(N_k)_{k\geq1}$ is a Cauchy sequence in $\mathfrak{A}$ as desired.
\end{proof}
Furthermore, we note that the following can be applied to every UHF C$^*$-algebra by Theorem \ref{UHFnil}.
\begin{thm}[Theorem 5.8 in \cite{Sk}]
\label{movediskaround}
Let $\Omega$ be a non-empty, open, connected and simply connected subset of $\mathbb{C}$ containing zero such that $\partial \Omega$ contains at least two points and is a Jordan curve.  Let $\mathfrak{A}$ be a C$^*$-algebra and suppose that there exists an $N \in Nor(\mathfrak{A}) \cap \overline{Nil(\mathfrak{A})}$ such that $\sigma(N)$ is the closed unit disk.  Then there exists an operator $N_0 \in Nor(\mathfrak{A}) \cap \overline{Nil(\mathfrak{A})}$ with $\sigma(N_0) = \overline{\Omega}$.
\end{thm}
To conclude this section, we will demonstrate that Corollary 6 in \cite{He2} cannot be generalized to AF C$^*$-algebras (it was demonstrated in Section 8 of \cite{Sk} that Corollary 9 in \cite{He2} cannot be generalized to C$^*$-algebras with faithful tracial states).  It is the existence of faithful tracial states on finite dimensional C$^*$-algebras that prevent the generalization of Herrero's result.
\begin{lem}
\label{shiftsinAF}
Let $\mathfrak{A}$ be a unital AF C$^*$-algebra and let $T \in \mathfrak{A}$.  Then each of the following sets is either the empty set or a singleton:
\begin{enumerate}
	\item $\left\{\lambda \in \mathbb{C} \, \mid \, \lambda I_\mathfrak{A} + T \in \overline{Nil(\mathfrak{A})}\right\}$.
	\item $\left\{\lambda \in \mathbb{C} \, \mid \, \lambda I_\mathfrak{A} + T \in \overline{span(Nil(\mathfrak{A}))}\right\}$.
\end{enumerate}
\end{lem}
\begin{proof}
We shall only prove the first claim since the proof of the second claim is exactly the same.  Suppose 
\[
\lambda_0 \in \left\{\lambda \in \mathbb{C} \, \mid \, \lambda I_\mathfrak{A} + T \in \overline{Nil(\mathfrak{A})}\right\}
\]
and let $R := \lambda_0 I_\mathfrak{A} + T$.  Thus to show that $\lambda I_\mathfrak{A} + T \notin \overline{Nil(\mathfrak{A})}$ for all $\lambda \in \mathbb{C} \setminus \{\lambda_0\}$ it suffices to show that $\mu I_\mathfrak{A} + R \notin \overline{Nil(\mathfrak{A})}$ for all $\mu \in \mathbb{C} \setminus \{0\}$.
\par 
Since $\mathfrak{A}$ is a unital AF C$^*$-algebra, $\mathfrak{A} = \overline{\bigcup_{k \geq 1} \mathfrak{A}_k}$ where $\mathfrak{A}_1 \stackrel{\alpha_1}{\rightarrow} \mathfrak{A}_2 \stackrel{\alpha_2}{\rightarrow} \mathfrak{A}_3 \stackrel{\alpha_3}{\rightarrow} \cdots$ is a direct limit of finite dimensional C$^*$-algebras with $\alpha_k$ unital and injective for all $k \in \mathbb{N}$.  Therefore there exists $R_k \in \mathfrak{A}_k$ such that $R = \lim_{k\to \infty} R_k$. However, since $R \in \overline{Nil(\mathfrak{A})}$, Proposition \ref{naqniAF} implies that $R = \lim_{k\to\infty} M_k$ where $M_k \in Nil(\mathfrak{A}_k)$ for all $k \in \mathbb{N}$.  Hence $\lim_{k\to\infty}\left\| R_k - M_k\right\| = 0$.  Thus $\lim_{k\to\infty} tr_{\mathfrak{A}_k}(R_k) = 0$ (where $tr_{\mathfrak{A}_k}$ is any faithful tracial state on $\mathfrak{A}_k$) as every nilpotent matrix has zero trace. 
\par
Fix $\mu  \in \mathbb{C}\setminus \{0\}$.   Then $\mu I_\mathfrak{A} + R = \lim_{k\to\infty} \mu I_{\mathfrak{A}_k} + R_k$.  If $\mu I_\mathfrak{A} + R \in \overline{Nil(\mathfrak{A})}$ then the above argument implies that $\lim_{k\to\infty} tr_{\mathfrak{A}_k}(\mu I_{\mathfrak{A}_k} + R_k) = 0$ which is impossible as $\mu \neq 0$ and $\lim_{k\to\infty} tr_{\mathfrak{A}_k}(R_k) = 0$. 
\end{proof}
\begin{cor}
Let $\mathfrak{A}$ be a unital AF C$^*$-algebra.  Then 
\[
I_\mathfrak{A} \notin \overline{span(Nil(\mathfrak{A}))}.
\]
\end{cor}
\begin{proof}
Note $0 \in \overline{span(Nil(\mathfrak{A}))}$ and apply Lemma \ref{shiftsinAF}.  
\end{proof}

\section{C$^*$-Algebras with Dense Subalgebras of Nilpotent Operators}
\label{sec:READSGEN}

In \cite{Re}, Read gave an example of a separable C$^*$-algebra that contains a dense subalgebra consisting entirely of nilpotent operators.  In this section we will use Lemma \ref{mas} and the construction in \cite{Re} to construct an approximately homogeneous (and thus separable, nuclear, and quasidiagonal) C$^*$-algebra that contains dense subalgebra consisting entirely of nilpotent operators.  It will also be demonstrated that there exists an AF C$^*$-algebra with a C$^*$-subalgebra $\mathfrak{D}$ where $\mathfrak{D} = \overline{Nil(\mathfrak{D})}$.  Thus the study of the closure of nilpotent operators in AF C$^*$-algebras is incredibly complex.
\begin{cons}
By Lemma \ref{mas} there exists finite dimensional Hilbert spaces $\{\mathcal{H}_n\}_{n\geq1}$, positive matrices $A_n \in \mathcal{B}(\mathcal{H}_n)$ of norm one, and nilpotent matrices $M_n \in \mathcal{B}(\mathcal{H}_n)$ such that $\sum_{n\geq1} \left\|A_n - M_n\right\| < \infty$.  Since each $A_n$ is a positive matrix with norm one, there exists unit vectors $\xi_n \in \mathcal{H}_n$ such that $A_n\xi_n =\xi_n$ for all $n \in \mathbb{N}$.
\par
We will use $\{\mathcal{H}_n\}_{n\geq1}$ and $\{\xi_n\}_{n\geq1}$ to generalize Read's construction.  Consider the sequence of pointed Hilbert spaces $(\mathcal{H}_n, \xi_n)$.  For each $n < m$ define $\phi_{n,m} : \otimes^n_{k=1} \mathcal{H}_k \to \otimes^m_{k=1} \mathcal{H}_k$ such that
\[
\phi_{n,m}(\eta_1 \otimes \eta_2 \otimes \cdots \otimes \eta_n) = \eta_1 \otimes \eta_2 \otimes \cdots \otimes \eta_n \otimes \xi_{n+1} \otimes \xi_{n+2} \otimes \cdots \otimes \xi_m.
\]
Let $\mathcal{K} := \otimes^\infty_{k=1} \mathcal{H}_k$ be the completion of the direct limit of the nested sequence of Hilbert spaces $\otimes^n_{k=1} \mathcal{H}_k$ with the connecting maps $\phi_{n,m}$.  Since each $\mathcal{H}_k$ is separable, each $\otimes^n_{k=1} \mathcal{H}_k$ is separable and thus $\mathcal{K}$ is separable.  Let $\phi_n : \otimes^n_{k=1} \mathcal{H}_k\to\mathcal{K}$ be the natural inclusion.
\par
We will maintain the above notation throughout the rest of this section.
\end{cons}
The following are new versions of Lemma 0.3 and Corollary 0.4 in \cite{Re} respectively that will serve our purposes.  We omit the proofs as they follow as in \cite{Re}.
\begin{lem}
\label{lem0.3}
Let $(S_n)_{n\geq1}$ be a sequence of operators with $S_n \in \mathcal{B}(\mathcal{H}_n)$ such that 
\[
C := \prod_{n\geq1} \max\{\left\|S_n\right\|, 1\} < \infty \,\,\,\,\,\mbox{ and }\,\,\,\,\,\sum_{n\geq1} \left\|S_n\xi_n - \xi_n\right\| < \infty.
\]
Then there exists a unique operator $S' \in \mathcal{B}(\mathcal{K})$ such that $S'(\phi_n \zeta) = S'_n \zeta$ for each $\zeta \in \otimes^n_{k=1} \mathcal{H}_k$ where $S'_n := \lim_{m\to\infty} S'_{n,m}$ where, for each $m > n$, $S'_{n,m} : \otimes^n_{k=1} \mathcal{H}_k \to \mathcal{K}$ is defined by
\[
S'_{n,m} = \phi_m \circ (\otimes^m_{i=1} S_i) \circ \phi_{n,m}.
\]
We will use $\otimes^\infty_{n=1} S_n$ to denote $S'$.
\end{lem}
\begin{cor}
\label{tensorapprox}
Let $S' = \otimes^\infty_{n=1} S_n$ and $R' = \otimes^\infty_{n=1} R_n$ be elements of $\mathcal{B}(\mathcal{K})$ as constructed in Lemma \ref{lem0.3}.  Then 
\[
\left\|S' - R'\right\| \leq C_S C_R \sum_{n\geq1} \left\|S_n - R_n\right\|
\]
where 
\[
C_S := \prod_{n\geq1} \max\{1, \left\|S_n\right\|\} \,\,\,\,\,\mbox{ and }\,\,\,\,\, C_R := \prod_{n\geq1} \max\{1, \left\|R_n\right\|\}.
\]
\end{cor}
\begin{cons}
\label{readcon}
Let $\mathfrak{B}$ be the C$^*$-subalgebra of $\mathcal{B}(\mathcal{K})$ generated by all operators of the form $\otimes^\infty_{n=1} S_n$ given by Lemma \ref{lem0.3}. Let $\mathcal{E}_A$ be subset of $\mathfrak{B}$ containing all operators of the form $\otimes^\infty_{n=1} S_n$ from Lemma \ref{lem0.3} such that there exist a $k \in \mathbb{N}$ such that $S_n = A_n$ for each $n\geq k$.  Since $\sum_{n\geq1} \left\|A_n\xi_n - \xi_n\right\| =0$ and $\left\|A_n\right\| = 1$ for all $n \in \mathbb{N}$, $\mathcal{E}_A$ is non-empty.  Let $\mathfrak{C}$ be the C$^*$-algebra generated by $\mathcal{E}_A$.  Note that $\mathcal{E}_A$ is a self-adjoint set so $\mathfrak{C}$ is the closure of the algebra generated by $\mathcal{E}_A$. 
\end{cons}
\begin{lem}
\label{propertiesofread}
The C$^*$-algebra $\mathfrak{C}$ from Construction \ref{readcon} is nuclear, quasidiagonal, approximately homogeneous, and separable.
\end{lem}
\begin{proof}
For each $k \in \mathbb{N}$ let $\mathfrak{C}_k$ be the C$^*$-subalgebra of $\mathfrak{C}$ generated by all operators of the form $\otimes^\infty_{n=1} S_n$ from Lemma \ref{lem0.3} such that $S_n = A_n$ for all $n > k$.  Then $\mathfrak{C}_k$ is isomorphic to $\mathcal{B}(\mathcal{H}_1) \otimes_{\min} \cdots \otimes_{\min} \mathcal{B}(\mathcal{H}_k) \otimes_{\min} \mathfrak{A}_{k+1}$ where $\mathfrak{A}_{k+1}$ is the abelian C$^*$-algebra generated by the infinite tensor $\otimes^\infty_{n=k+1} S_n$ where $S_n = A_n$ for all $n> k$.  Since $\mathfrak{C}$ is the inductive limit of the C$^*$-algebras $\mathfrak{C}_k$, the result follows. 
\end{proof}
\begin{thm}
\label{readnil}
The C$^*$-algebra $\mathfrak{C}$ from Construction \ref{readcon} has a dense subalgebra $\mathfrak{N}$ such that every operator of $\mathfrak{N}$ is nilpotent.
\end{thm}
\begin{proof}
This proof is nearly identical to that of Theorem 1.2 in \cite{Re} where the only changes are our simple modifications.  Let $\mathcal{E}_N$ be the subset of $\mathfrak{B}$ consisting of all operators of the form $\otimes^\infty_{n=1} S_n$ from Lemma \ref{lem0.3} such that there exist a $k \in \mathbb{N}$ such that $S_n = M_n$ for all $n\geq k$.  Let $\mathfrak{N}$ be the (not necessarily closed nor self-adjoint) subalgebra of $\mathfrak{B}$ generated by $\mathcal{E}_N$.  It suffices to show three things: (1) $\mathfrak{N} \subseteq \mathfrak{C}$; (2) $\mathfrak{N}$ is dense in $\mathfrak{C}$; (3) every operator of $\mathfrak{N}$ is nilpotent.
\par
Proof of (1):  It suffices to show that $\mathcal{E}_N \subseteq \mathfrak{C}$.  To begin we will show that $\mathfrak{N}$ is not empty.  Suppose that $(S_n)_{n\geq1}$ is a sequence of operators where $S_n \in \mathcal{B}(\mathcal{H}_n)$ for all $n\in\mathbb{N}$ and $S_n = M_n$ for all $n \geq k$ for some fixed $k \in \mathbb{N}$.  Since $\left\|A_n\right\| = 1$ for all $n\in \mathbb{N}$ and $\sum_{n\geq1} \left\|M_n - A_n\right\| < \infty$, $\sum_{n\geq1} |\left\|S_n\right\| - 1| < \infty$ and so $\prod_{n\geq1} \max\{1, \left\|S_n\right\|\} < \infty$.  Moreover, since $\sum_{n\geq1} \left\|A_n\xi_n - \xi_n\right\| = 0$,
\[
\sum_{n\geq k} \left\|S_n \xi_n - \xi_n\right\| \leq \sum_{n\geq k} \left\|M_n - A_n\right\|  + \sum_{n\geq k}  \left\|A_n\xi_n - \xi_n\right\| < \infty.
\]
Hence Lemma \ref{lem0.3} implies that the operator $\otimes^\infty_{n=1} S_n$ exists.  Hence $\mathfrak{N}$ is not empty.
\par
Fix a sequence $(S_n)_{n\geq1}$ of operators where $S_n \in \mathcal{B}(\mathcal{H}_n)$ for all $n\in\mathbb{N}$ and $S_n = M_n$ for all $n \geq k$.  For each $m \geq k$ define $R_m := \left(\otimes^m_{n=1} S_n\right) \otimes \left(\otimes^\infty_{n=m+1} A_n\right)$.  Then $\{R_m\}_{m\geq k} \subseteq \mathcal{E}_A$ by construction and, by Corollary \ref{tensorapprox},
\[
\left\|\otimes^\infty_{n=1} S_n - R_m\right\| \leq \left(\prod_{n\geq1} \max\{\left\|S_n\right\|, 1\}\right)^2 \sum_{n\geq m+1} \left\|A_n - M_n\right\|.
\]
Therefore, since $\lim_{m\to\infty} \sum_{n\geq m+1} \left\|A_n - M_n\right\|= 0$, $\otimes^\infty_{n=1} S_n$ is in the closure of $\{R_m\}_{m\geq k}$ and thus $\otimes^\infty_{n=1} S_n \in \mathfrak{C}$.  Hence $\mathfrak{N}\subseteq \mathfrak{C}$ as desired.
\par
Proof of (2): It suffices to show that $\mathcal{E}_A$ is in the closure of $\mathfrak{N}$ since $\mathfrak{C}$ is the closure of the algebra (and not $^*$-algebra) generated by $\mathcal{E}_A$.  Fix an operator $T:= \left(\otimes^k_{n=1} S_n\right) \otimes \left(\otimes^\infty_{n=k+1} A_n\right) \in \mathcal{E}_A$.  For each $m \geq k$ let $R_m := \left(\otimes^m_{n=1} S_n\right) \otimes \left(\otimes^\infty_{n=m+1} M_n\right)$.  Then $\{R_m\}_{m\geq k} \subseteq \mathfrak{N}$ and, by Corollary \ref{tensorapprox}, $\left\|T - R_m\right\|$ is at most
\[
\left( \prod^k_{n=1} \max\{\left\|S_n\right\|, 1\}\right)^2 \left( \prod_{n\geq1} \max\{\left\|M_n\right\|, 1\}\right) \sum_{n\geq m+1} \left\|A_n - M_n\right\|.
\]
Therefore, since $\lim_{m\to\infty} \sum_{n\geq m+1} \left\|A_n - M_n\right\|= 0$, $T \in \overline{\mathfrak{N}}$.  Hence $\mathcal{E}_A$ is in the closure of $\mathfrak{N}$ so $\mathfrak{N}$ is dense in $\mathfrak{C}$.
\par
Proof of (3): Notice that every operator $N$ of $\mathfrak{N}$ can be written in the form
\[
N = \sum^\ell_{k=1} S_k \otimes \left(\otimes^\infty_{i=n+1} M_i^k\right)
\]
for some $n, \ell \in \mathbb{N}$ and $S_1,\ldots, S_\ell \in \mathcal{B}\left(\otimes^n_{k=1} \mathcal{H}_k\right)$.  Therefore, since there exists an $m_{n+1} \in \mathbb{N}$ such that $M^{m_{n+1}}_{n+1} = 0$, $N^{m_{n+1}} = 0$ (by the trivial computation that $\left(\otimes^\infty_{n=1} R_n\right)\left(\otimes^\infty_{n=1} R'_n\right) = \otimes^\infty_{n=1} R_nR'_n$).  Hence $N$ is nilpotent so every operator of $\mathfrak{N}$ is nilpotent.  
\end{proof}
One interesting consequence is the following which is quite surprising since every other C$^*$-algebra $\mathfrak{A}$ with $Nor(\mathfrak{A}) \cap \overline{Nil(\mathfrak{A})} \neq \{0\}$ studied in this paper and in \cite{Sk} has had a plethora of projections.
\begin{cor}
Let $\mathfrak{C}$ be the C$^*$-algebra from Construction \ref{readcon}.  Then $\sigma(T)$ is connected and contains zero for all $T \in \mathfrak{C}$.  Thus $\mathfrak{C}$ is projectionless.
\end{cor}
\begin{proof}
The result is trivial by Theorem \ref{readnil} and Lemma \ref{cac0}.
\end{proof}
To conclude this section we will demonstrate that there exists an AF C$^*$-algebra that contains a C$^*$-subalgebra $\mathfrak{D}$ such that $\mathfrak{D} = \overline{Nil(\mathfrak{D})}$.  This demonstrates that the study of the closure of the nilpotent operators in an AF C$^*$-algebra is incredibly complex.  To begin we note the following trivial observation from the proof of Theorem \ref{readnil}.
\begin{lem}
\label{nilalgproduct}
Let $\mathfrak{C}$ be the C$^*$-algebra from Construction \ref{readcon}, let $\mathfrak{N}$ be the subalgebra of $\mathfrak{C}$ from Theorem \ref{readnil}, and $N_1, \ldots, N_m \in \mathfrak{N}$.  Then there exists an $\ell \in \mathbb{N}$ (depending on $N_1, \ldots, N_m$) such that $N_{n_1} N_{n_2} \cdots N_{n_\ell} = 0$ for any selection of $n_j \in \{1,\ldots, m\}$.
\end{lem}
\begin{proof}
This result is trivial by the structure of elements of $\mathfrak{N}$ from the third part of the proof of Theorem \ref{readnil}. 
\end{proof}
\begin{lem}
\label{readsubalg}
Let $\mathfrak{C}$ be the C$^*$-algebra from Construction \ref{readcon} and let $\mathfrak{N}$ be the subalgebra of $\mathfrak{C}$ from Theorem \ref{readnil}.  The subalgebra 
\[
C_0(0,1] \odot \mathfrak{N} := \left\{ \sum^m_{j=1} f_j \otimes N_j \, \mid \, m \in \mathbb{N}, N_j \in \mathfrak{N}, f_j \in C_0(0,1]\right\}
\]
of $C_0(0,1] \otimes_{\min} \mathfrak{C}$ is dense and consists entirely of nilpotent operators.
\end{lem}
\begin{proof}
Clearly $C_0(0,1] \odot \mathfrak{N}$ is a dense subalgebra of $C_0(0,1] \otimes_{\min} \mathfrak{C}$ as $\mathfrak{N}$ is a dense subalgebra of $\mathfrak{C}$.  Let $\sum^m_{j=1} f_j \otimes  N_j \in C_0(0,1] \odot \mathfrak{N}$ be arbitrary.  By Lemma \ref{nilalgproduct} there exists an $\ell \in \mathbb{N}$ such that $N_{n_1} N_{n_2} \cdots N_{n_\ell} = 0$ for any $n_j \in \{1,\ldots, m\}$.  Thus $\left(\sum^m_{j=1} f_j \otimes N_j\right)^\ell = 0$ so every element of $C_0(0,1] \odot \mathfrak{N}$ is nilpotent. 
\end{proof}
\begin{prop}
\label{AFwithdensenil}
There exists an AF C$^*$-algebra $\mathfrak{A}$ and a C$^*$-subalgebra $\mathfrak{D}$ of $\mathfrak{A}$ such that $\mathfrak{D}$ has a dense subalgebra consisting entirely of nilpotent operators.
\end{prop}
\begin{proof}
Let $\mathfrak{C}$ be the C$^*$-algebra from Construction \ref{readcon} and let 
\[
\mathfrak{D} := C_0(0,1] \otimes_{\min} \mathfrak{C}.
\]
Then $\mathfrak{D}$ is AF-embeddable by Lemma \ref{propertiesofread} and by Proposition 2 in \cite{Oz}.  Thus the result follows from Lemma \ref{readsubalg}.
\end{proof}

\section*{Acknowledgements} 
This research was supported in part by NSERC PGS.  The author would like to thank Laurent Marcoux for informing me of this problem, his proof of Theorem \ref{UHFnil}, his idea that tracial states provide restrictions to positive limits of nilpotent operators, for the multiple discussions with him pertaining to this problem, and for his comments and advice on this paper.

\end{document}